\PassOptionsToPackage{dvipsnames}{xcolor}
\documentclass[11pt]{amsart}
\usepackage{amsmath,amssymb,amsfonts}

\author[D.\ Chan]{David Chan}
\address[Chan]{
Department of Mathematics,
         Michigan State University,
         East Lansing, MI
         U.S.A.}
\email{\href{mailto:chandav2@msu.edu}{chandav2@msu.edu}}

\keywords{Tambara functors, Nakayama Lemma}

\subjclass[2020]{55P91, 
19A22, 
13A15 
}

\usepackage[margin=1.25in]{geometry}

\usepackage{quiver} 
\usepackage{mathrsfs} 
\usepackage{mathtools} 
\usepackage{xcolor} 
\usepackage[shortlabels]{enumitem} 
\usepackage{microtype} 

\usepackage[unicode=true, pdfusetitle,
    colorlinks=true,
    linkcolor=orange,
    citecolor=orange,
    urlcolor=orange
]{hyperref} 

\usepackage[nameinlink]{cleveref}

\DeclareMathOperator{\res}{res}
\DeclareMathOperator{\tr}{tr}
\DeclareMathOperator{\nm}{nm}

\DeclareMathOperator{\Spec}{Spec}

\newcommand{\mf}[1]{\mathfrak{#1}}
\newcommand{\ul}[1]{\underline{#1}}
\newcommand{\ulZ}{\ul{\mathbb{Z}}}
\newcommand{\mult}{\mathrm{mult}}
\newcommand{\ol}[1]{\overline{#1}}
\newcommand{\tamb}{\mathrm{Tamb}}
\newcommand{\Mack}{\mathrm{Mack}}
\newcommand{\Mod}{\mathrm{Mod}}

\newcommand{\bbF}{\mathbb{F}}
\newcommand{\bbZ}{\mathbb{Z}}

\numberwithin{equation}{section} 
\numberwithin{figure}{section}
\crefname{lemma}{Lemma}{Lemmas}
\crefname{theorem}{Theorem}{Theorems}
\crefname{lettertheorem}{Theorem}{Theorems}
\crefname{definition}{Definition}{Definitions}
\crefname{proposition}{Proposition}{Propositions}
\crefname{remark}{Remark}{Remarks}
\crefname{corollary}{Corollary}{Corollaries}
\crefname{equation}{Equation}{Equations}
\crefname{construction}{Construction}{Constructions}
\crefname{ex}{Example}{Examples}
\crefname{appsec}{Appendix}{Appendices}
\crefname{section}{Section}{Sections}
\crefname{subsection}{Subsection}{Subsections}

\theoremstyle{plain}
\newtheorem{theorem}[equation]{Theorem}

\newtheorem{corollary}[equation]{Corollary}
\newtheorem{proposition}[equation]{Proposition}
\newtheorem{lemma}[equation]{Lemma}
\newtheorem*{theorem*}{Theorem}

\theoremstyle{definition}
\newtheorem{definition}[equation]{Definition}
\newtheorem{example}[equation]{Example}
\newtheorem{remark}[equation]{Remark}

\usepackage[
    maxbibnames=99,
    isbn=false,
    url=true,
    doi=false,
    backref=true
]{biblatex}
\DefineBibliographyStrings{english}{%
  backrefpage = {Cited on page},
  backrefpages = {Cited on pages},
}
\DeclareFieldFormat{url}{\url{#1}} 

\usepackage[
    textwidth=3cm,
    colorinlistoftodos]
{todonotes} 

\theoremstyle{plain}
\newtheorem{introtheorem}{Theorem}
 
\crefname{introtheorem}{Theorem}{Theorems}

\title{Multiplicatively cohomological Tambara functors}

\date{}

\begin{document}

\maketitle
\begin{abstract}
    We study a nice class of Tambara functors, called multiplicatively cohomological Tambara functors.  The first result is a proof of a conjecture of Balchin, Quigley, and Spitz showing that every multiplicatively cohomological Tambara functor is also additively cohomological.  Building on this, we prove Nakayama's lemma for multiplicatively cohomological Tambara functors. As an application, we identify a condition under which finitely generated projective modules over local Tambara functors are always free.  This applies, in particular, to the constant $C_{p^n}$-Tambara functor $\ul{\bbZ}_p$.
\end{abstract}

\section{Introduction}

Tambara functors are an important algebraic structure in equivariant stable homotopy theory. While the original motivating example for studying Tambara functors was (even degree) group cohomology of a finite group $G$ \cite{tambara:1993}, other examples arise naturally in representation theory, algebraic $K$-theory, and motivic homotopy theory \cite{calle/ginnett:2019,ElmantoHaugseng,Bachmann}.  The connection to stable homotopy theory comes from the fact that multiplicative equivariant cohomology theories naturally admit the structure of (graded) Tambara functors \cite{AngeltveitBohmann}.


A $G$-Tambara functor $T$ consists of a system of commutative rings $T(G/H)$, indexed on the subgroups $H\leq G$, together with a series of operations connecting these rings.  The operations come in four distinct flavors:  for $H\leq K\leq G$ and $g\in G$ we have 
\begin{enumerate}
    \item (transfer) $\tr^K_H\colon T(G/H)\to T(G/K)$,
    \item (norm) $\nm^K_H\colon T(G/H)\to T(G/K)$,
    \item (restriction) $\res^K_H\colon T(G/K)\to T(G/H)$,
    \item (conjugation) $c_{g,H} \colon T(G/H)\to T(G/gHg^{-1})$,
\end{enumerate}
which are subject to several compatibility conditions.  
Tambara functors provide leverage in computations, and recent interest is largely motivated by their success in tackling problems in topology \cite{HHR,BH2015}. 

The Tambara functor operations in group cohomology satisfy an additional relation which is not an axiom for an arbitrary Tambara functor.  Specifically, for any $H\leq K\leq G$, the group cohomology Tambara functor satisfies the relation that the composite $\tr^K_H\circ \res^K_H$ is equal to multiplication by $|K/H|$. For this reason, any Tambara satisfying this relation is called \emph{cohomological}.  The cohomological relation is remarkably helpful in applications.  For instance, it gives a one sentence proof of the fact that for $n>0$, the cohomology group $H^n(G,\mathbb{Z})$ is $|G|$-torsion. Indeed, multiplication by $|G|$ factors as
\[
    H^n(G,\mathbb{Z})\xrightarrow{\res^G_1} H^{n}(1;\mathbb{Z})\xrightarrow{\tr^G_{1}} H^n(G,\mathbb{Z})
\]
and $H^{n}(1;\mathbb{Z})=0$ for $n>0$.

In recent work studying the algebraic geometry of Tambara functors, Balchin, Quigley, and Spitz introduce the term \emph{multiplicatively cohomological}, to describe any Tambara functor satisfying a multiplicative variant of the cohomological relation \cite{BQS}.  Precisely, a Tambara functor is multiplicatively cohomological if for all $H\leq K\leq G$ the composite $\nm^K_H\circ \res^K_H$ is the $|K/H|$-th power map. To disambiguate the two terms, we will refer to Tambara functors satisfying the more standard cohomological condition as \emph{additively cohomological}.

In \cite{BQS}, the authors prove that when $G$ is a nilpotent group every multiplicatively cohomological $G$-Tambara functor is also additively cohomological.  They conjecture, furthermore, that this is true for all finite groups $G$ \cite[Conjecture 60]{BQS}.  The first main result of this paper is a confirmation of this conjecture.

\begin{introtheorem}\label{thm:main}
    Let $G$ be a finite group.  Any multiplicatively cohomological $G$-Tambara functor is also additively cohomological.
\end{introtheorem}

After proving the conjecture, we turn to a develop a bit more of the theory of multiplicatively cohomological Tambara functors.  While the analogy between Tambara functors and commutative rings is quite strong, there are times when taking one's intuition from commutative algebra is misleading.  One important place where this happens is in the study of localizations, where the theory is difficult and remains underdeveloped. 

One issue is that the analogs of several foundational results from commutative algebra actually fail when one starts to consider questions related to localization.  For instance, Nakayama's lemma is not true for Tambara functors, in general. That is, if $T$ is a Tambara functor we can can define the Jacobson radical $J\subset T$ as the intersection of maximal ideals, and there can exist non-zero $T$-modules $M$ with $JM=M$. We provide an explicit example demonstrating the failure of Nakayama's lemma in \Cref{ex: nak fails in general}.

Our second main result is that Nakayama's lemma holds for multiplicatively cohomological Tambara functors
\begin{introtheorem}[Nakayama's lemma]\label{introthm: nakayama}
    Let $T$ be a multiplicatively cohomological Tambara functor with Jacobson radical $J$.  If $M$ is a $T$-module such that $JM=M$ then $M=0$.
\end{introtheorem}

A classical application of Nakayama's lemma for commutative rings is that every finitely generated projective module over a local ring is free.  The proof amounts to the observation that Nakayama's lemma allows us to reduce the claim to that of fields, where every module is free. Working analogously, we prove the following result using \cref{introthm: nakayama}.

\begin{introtheorem}\label{thm C}
    Let $T$ be a local, multiplicatively cohomological Tambara functor with maximal ideal $\mf{m}$.  If $P$ is a finitely generated projective $T$-module such that $P/\mf{m}P$ is a free $T/\mf{m}$-module, then $P$ is a free $T$-module.
\end{introtheorem}
We warn the reader that our definition of free module is more general than a direct sum of copies of $T$.  See \cref{def: free modules} for details.

While $T/\mf{m}$ is a Tambara field, in the sense that it has no non-zero proper ideals, it is not known whether this guarantees that all projective modules are free. Some sufficient conditions on $T/\mf{m}$ to guarantee that all projective modules are free are identified in \cite[Theorem E]{ChanWisdom}, recalled below as \cref{thm: Chan Wisdom}. In particular, this holds for the constant $C_{p^n}$-Tambara functors at a field. Using this, we can immediately identify several cases of local Tambara functors where every finitely generated projective is free.
\begin{introtheorem}\label{thm D}
    If $\mathcal{O}$ is a commutative local ring, then the constant $C_{p^n}$-Tambara functor at $\mathcal{O}$ is local, and every finitely generated projective $\mathcal{O}$-module is free.
\end{introtheorem}

While a result like this may seem basic, a direct proof without some version of Nakayama's lemma seems rather difficult to produce.  As a future direction, it would be interesting to expand the class of Tambara fields for which every finitely generated projective module is free. The results here would the immediately allow us to lifts such results to more local Tambara functors.

\noindent \textbf{Organization:} The rest of the paper is organized as follows. All of the background material is recalled in \cref{sec: background}.  The proof of \cref{thm:main} is carried out in \cref{sec: reduction,sec: 4}.  Nakayama's lemma for multiplicatively cohomological Tambara functors is proven in \cref{sec: nakayama}.  Finally, in \cref{sec: local} we give the applications to local Tambara functors, including the proofs of \cref{thm C,thm D}.

\subsection{Acknowledgments}
I would like to thank Ben Spitz for pointing out the conjecture in \cite{BQS} and for several helpful conversations. Thanks also to Maxine Calle and Teena Gerhardt for discussions which improved the exposition.  The author was partially supported by NSF grant DMS-2135960.

\section{Tambara functors}\label{sec: background}

    In this section we recall the necessary background on Tambara functors,, particularly related to ideals and modules.  We will not give a comprehensive review and instead will simply recall the facts that we need. We refer the reader to any of \cite{CCMQSV,mazur:2019,strickland:2012} for a more thorough treatment.
    
    Let $G$ be a finite group. As recalled in the introduction, a $G$-Tambara functor consists of a collection of commutative rings $T(G/H)$, together with transfer, norm, restriction, and conjugation operations subject to some compatibility conditions.  A map of Tambara functors $f\colon S\to T$ is a collection of a ring homomorphisms $f_{G/H}\colon S(G/H)\to T(G/H)$ which are commute with the four operations. For the trivial group $G=e$, a Tambara functor is the same a commutative ring. Below, we give a (non-exhaustive) list of compatibility conditions that we will use in this paper.

\begin{proposition}\label{prop: compatibility}
    The operations in a $G$-Tambara functor satisfy the following relations.
    \begin{enumerate}
        \item The transfer maps are additive, the norm maps are multiplicative. The restriction and conjugation maps are ring maps.
        \item (associativity) For any $H\leq K\leq L$ and $g,x\in G$ we have $\tr^L_K\tr^K_H = \tr^L_H$, $\nm^L_K\nm^K_H = \nm^L_H$,  $\res^K_H\res^L_K = \res^L_H$, and $c_{x,gHg^{-1}}c_{g,H} = c_{xg,H}$,
        \item (Frobenius reciprocity) For any $x\in T(G/K)$ and $y\in T(G/H)$ we have $x\cdot \tr^K_H(y) = \tr^K_H(\res^K_H(x)\cdot y)$.
        \item (conjugation commutes) For any $H\leq K$ and $g\in G$ we have $\res^{gKg^{-1}}_{gHg^{-1}}c_{g,K} = c_{g,H}\res^K_H$, $\tr^{gKg^{-1}}_{gHg^{-1}}c_{g,H} = c_{g,H}\tr^K_H$, and $\nm^{gKg^{-1}}_{gHg^{-1}}c_{g,H} = c_{g,H}\nm^K_H$,
        \item  For any $K\leq G$, and $k\in K$ we have $c_{k,K} = \mathrm{id}_{T(G/K)}$.
        \item (additive double coset formula) For any $L,H\leq K$ we have 
        \[
            \res^K_L\tr^K_H = \sum\limits_{\gamma\in L\backslash K/H} \tr^L_{L\cap \gamma H\gamma^{-1}} \res^{\gamma H\gamma^{-1}}_{L\cap \gamma H\gamma^{-1}} c_{\gamma,H}
        \]
        \item (multiplicative double coset formula) For any $L,H\leq K$ we have 
        \[
            \res^K_L\nm^K_H = \prod\limits_{\gamma\in L\backslash K/H} \nm^L_{L\cap \gamma H\gamma^{-1}} \res^{\gamma H\gamma^{-1}}_{L\cap \gamma H\gamma^{-1}} c_{\gamma,H}
        \]
    \end{enumerate}
\end{proposition}

\begin{remark}\label{remark: Weyl group actions}
    For any $K\leq G$ and $g\in N_G(K)$ the conjugation map $c_{g,K}$ is a ring automorphism of $T(G/K)$.  These assemble into an action of the Weyl group $W_G(K):=N_G(K)/K$ on $T(G/K)$, by (5) from \cref{prop: compatibility}. In particular, $T(G/e)$ admits an action by $G$ and relation (4) above implies that the image of the restriction map $\res^K_e\colon T(G/K)\to T(G/e)$ lands the the $K$-fixed points $T(G/e)^K$. For $g\in N_G(K)$ and $x\in T(G/K)$ we will often simply write $g\cdot x$ instead of $c_{g,K}(x)$. 
\end{remark}

We start with two important examples.
\begin{example}
    The Burnside Tambara functor, denoted $A_G$, has $A_G(G/H) = A(H)$, the Burnside ring of virtual finite $H$-sets.  The multiplication is given by cartesian product and the addition is given by disjoint union. The transfer and norm operations are induced by induction and coinduction of sets with group action.  The restriction and conjugation actions are given by the operation of the same name on sets with group action.
\end{example}
\begin{example}
    The constant Tambara functor $\ulZ$ is given by $\ulZ(G/H) = \mathbb{Z}$ for all $H$.  The transfer map is given by $\tr^K_H\colon \mathbb{Z}\to \mathbb{Z}$ is multiplication by $|K/H|$.  The norm $\nm^K_H\colon \mathbb{Z}\to \mathbb{Z}$ is the $|K/H|$-th power map. The restrictions and conjugations are identities.
\end{example}

We write $\tamb_G$ for the category of $G$-Tambara functors. For $H\leq K$, the categories $\tamb_H$ and $\tamb_K$ are related by an adjunction.

\begin{proposition}
    For $H\leq K$ there is a forgetful functor $R^K_H\colon \tamb_K\to \tamb_H$ with natural isomorphisms of rings $R^K_H(T)(H/L)\cong T(K/L)$ for all $L\leq H$.  Moreover, there is an adjunction $N^K_H\dashv R^K_H$, and isomorphisms $A_{K}\cong N^K_H(A_H)$ and $A_H\cong R^K_H(A_K)$.
\end{proposition}
\begin{proof}
    The first claim follows from  \cite[Remark 7.2]{Chan:biincomplete}.  The second follows from \cite[Lemma 7.3]{Chan:biincomplete}.
\end{proof}
\begin{corollary}
    The Burnside Tambara functor $A_G$ is the initial object of $\tamb_G$.
\end{corollary}
\begin{proof}
    A Tambara functor for the trivial group $e$ is the same as a commutative ring, and by definition $A_e$ is isomorphic to the integers, the initial commutative ring.  Since $A_G\cong N^G_e(A_e)$, by the preceding proposition, the claim follows from the fact the left adjoints preserve initial objects. 
\end{proof}

\subsection{Ideals and cohomological properties}

In this subsection we recall the definitions of ideals, and additively and multiplicatively cohomological Tambara functors. Ideals of Tambara functors were introduced by Nakaoka in \cite{nakaoka:2012}. 

\begin{definition}
      An \emph{ideal} $I$ of a Tambara functor $T$ is any collection of ring ideals $I(G/H)\subseteq T(G/H)$ which occurs as the levelwise kernels of a map of Tambara functors $f\colon T\to S$. These ideals are automatically closed under transfer, norm, restriction, and conjugation.
\end{definition}

  \begin{example}\label{example: kernel of eta}
    The initial map $\eta\colon A_G\to \ulZ$ is given levelwise by 
    \[
        \eta_{G/H}= \res^{H}_e\colon A(H)\to A(e)=\mathbb{Z}.
    \]
    Equivalently, this is the map which sends every virtual finite $H$-set $[X]-[Y]$ to the integer $|X|-|Y|$.  The kernel of this map is sometimes called the \textit{augmentation ideal} of $A_G$, and we denote it by $I_G\subset A_G$.  It is an exercise in the orbit stabilizer theorem to show that $I_G(G/K)$ is generated by elements of the form $|K/H|-K/H$, where $H\leq K$.
\end{example}

The generators of the augmentation ideal make it so that in $\ul{\bbZ}$ there is no distinction between $G/K$ and $|G/K|$.  Phrased another way, the relation
\[
    G/K = \tr^G_K\res^G_K(1) 
\]
in $A_G$ becomes $|G/K|=\tr^G_K\res^G_K(1)$ in $\ul{\bbZ}$.  Similarly, one checks that in $\ul{\bbZ}(G/G)$ we have the relation $n^{|G/K|} = \nm^G_K\res^G_K(n)$ for all $n\in \bbZ$.  The next definition turns these relations into definitions.

\begin{definition}
    Let $T$ be a $G$-Tambara functor.  We say that $T$ is:
    \begin{itemize}
        \item \emph{additively cohomological} if $\tr^K_H\circ \res^K_H$ is multiplciation by $|K/H|$ for any $H\leq K\leq G$;
        \item \emph{multiplicatively cohomological} if $\nm^K_H\circ \res^K_H$ is the $|K/H|$-th power map for any $H\leq K\leq G$.
    \end{itemize}
\end{definition}
\begin{example}
    The Tambara functor $\ulZ$ is both additively and multiplicatively cohomological.  The Burnside Tambara functor $A_G$ is neither.
\end{example}
\begin{example}[{\cite[(3.4)]{tambara:1993}}]
    For a finite group $G$, there is a Tambara functor  $\mathcal{H}$ given by
    \[
        \mathcal{H}(G/K)=\bigoplus_{n\geq 0} H^{2n}(K;\mathbb{Z})
    \]
    where the transfer, restriction, and conjugation are given by the operations of the same name in group cohomology.  The norm is given by the Evens norm \cite{Evens}.  The Tambara functor $\mathcal{H}$ is additively cohomological, which is the reason for the terminology.  We note that $\mathcal{H}$ is \textit{not} multiplicatively cohomological in general, which can be see when $G = C_2$.  There, $\mathcal{H}(C_2/C_2) \cong \mathbb{Z}[\beta]/(2\beta)$, where $\beta\in H^2(C_2;\mathbb{Z})$ is the polynomial generator.  Since $\beta^2\neq 0$, but $\res^{C_2}_e(\beta)\in H^2(e;\mathbb{Z})=0$ this Tambara functor is not multiplicatively cohomological.
\end{example}

\begin{example}\label{example: add not mult}
    We give an explicit example of a Tambara functor which is additively cohomological but not multiplicatively cohomological.  Let $G=C_2$ and let $T$ be the $C_2$-Tambara functor given by the \emph{Lewis diagram}
\[\begin{tikzcd}[ampersand replacement=\&]
	{\bbF_2[x,n]/(x^4-n^2)} \\
	\\
	{\bbF_2[x]}
	\arrow["{\res^{C_2}_e}"{description}, from=1-1, to=3-1]
	\arrow["{0=\tr^{C_2}_e}", shift left=3, curve={height=-12pt}, from=3-1, to=1-1]
	\arrow["{\nm^{C_2}_e}"', shift right=3, curve={height=12pt}, from=3-1, to=1-1]
\end{tikzcd}\]
where the top value is $T(C_2/C_2)$ and the bottom value is $T(C_2/e)$. The transfer is zero, $\res^{C_2}_e(x) = x$, $\res^{C_2}_e(n) = x^2$, and $\nm^{C_2}_e$ is the unique ring homomorphism with $\nm^{C_2}_{e}(x) = n$.
\end{example}

We end this subsection with a lemma we will need later.
\begin{lemma}\label{lemma: kernel of res is nilpotent}
    Suppose that $T$ is a multiplicatively cohomological $G$-Tambara functor, and let $x\in T(G/K)$ be an element such that $\res^K_H(x)=0$ for some $H<K$. Then $x$ is a nilpotent.
\end{lemma}
\begin{proof}
    We have $x^{|K/H|} = \nm^K_H\res^K_H(x) = \nm^K_H(0)=0.$
\end{proof}

\subsection{Mackey functors and modules}
Finally, we will need the definition of a module over a Tambara functor $T$.  For this, we need the definition of Mackey functors.

\begin{definition}
     A $G$-\emph{Mackey functor} $M$ consists of a collection of abelian groups $M(G/H)$ for each subgroup $H\leq G$ together with group homomorphisms $\tr^K_H, \res^K_H$ and $c_{g,H}$, as in the definition of a Tambara functor. A map of Mackey functors $f\colon M\to N$ is a collection of a group homomorphisms $f_{G/H}\colon M(G/H)\to N(G/H)$ which commute with the three operations.
\end{definition}

\begin{remark}
    The category of Mackey functors can be written as the category of product preserving, abelian group valued presheaves on a category $\mathcal{B}$, called the Burnside category.  Consequently, the category of Mackey functors is abelian.
\end{remark}

A Mackey functor is essentially a Tambara functor which has forgotten all multiplicative structure in sight.\footnote{This is not literally true.  There are Mackey functors which do not come from Tambara functors.}  In particular, it has forgotten that the levels are commutative rings, and it has forgotten the norm maps exist. Accordingly, every Tambara functor has an underlying Mackey functor. The transfer, restriction, and conjugation in a Mackey functor satisfy all relations from \cref{prop: compatibility} which do not make use of multiplication or norms.

The category of Mackey functors has a symmetric monoidal product, called the box product and denoted $\boxtimes$, and the unit is the Burnside Tambara functor.  Every Tambara functor is a commutative monoid for the box product, however there are commutative monoids which are not Tambara functors. Commutative monoids for the box product are called (commutative) Green functors. The box product admits an explicit description.

\begin{theorem}[{\cite[Theorem 5.5]{Sun}}]\label{thm: box product formula}
    Let $M$ and $N$ be two $G$-Mackey functors.  For all $H\leq G$ there is an isomorphism of abelian groups
    \[
        (M\boxtimes N)(G/K)\cong \left(\bigoplus_{H\leq K} M(G/H)\otimes N(G/H)\right)/I_K
    \]
    where $I_K$ is the submodule generated by all elements of the form
    \begin{enumerate}
        \item $\tr^L_H(x)\otimes y - x\otimes \res^L_H(y)$, for $x\in M(G/H)$ and $y\in N(G/L)$,
        \item $x\otimes \tr^L_H(y)-\res^L_H(x)\otimes y$ for $x\in M(G/L)$ and $y\in N(G/H)$,
        \item $c_{g,H}(x)\otimes c_{g,H}(y) - x\otimes y$ for $x\in M(G/H)$ and $y\in M(G/H)$ and $g\in K$,
    \end{enumerate}
    for $H\leq L\leq K$.
\end{theorem}

We will need the following corollary.

\begin{corollary}\label{cor: box product whe stuff vanishes}
    Let $M$ and $N$ be two $G$-Mackey functors.  Suppose $K\leq G$ is such that that $N(G/H)=0$ for all proper subgroups $H<K$. Then $(M\boxtimes N)(G/K)$ is the quotient of $M(G/K)\otimes N(G/K)$ by the relation that $\tr^K_H(x)\otimes y=0$ for all $x\in M(G/H)$ and all $H<K$.
\end{corollary}
\begin{proof}
    The assumption on $N$ implies that all terms in the direct sum of \cref{thm: box product formula} vanish except for the term coming from $H=K$.  Moreover, all the relations in $I_K$ are automatically zero, except for the first which becomes the relation in the statement.
\end{proof}

We can now turn to the theory of modules.

\begin{definition}
    A \textit{module} over a Tambara functor is a Mackey functor $M$, together with a map $\mu\colon T\boxtimes M\to M$ satisfying the usual associativity and unitality axioms. We write $\Mod_T$ for the abelian category of $T$-modules and morphisms between then. 
\end{definition}

\begin{example}\label{ex: algebras are modules}
    If $S\to T$ is a map of Tambara functors then the underlying Mackey functor of $T$ is an $S$-module via the action
    \[
        S\boxtimes T\to T\boxtimes T\to T
    \]
    where the second map is multiplication map for $T$.
\end{example}

A module over a Tambara functor $T$ can be reformulated as a Mackey functor $M$ such that for all $K\leq G$ the abelian group $M(G/K)$ is a module over $T(G/K)$, subject to various conditions relating the transfer, restrictions, conjugations, and multiplications. This observation seems to be due to Lewis, in unpublished work. We refer the reader to \cite[Lemma 2.14]{ChanWisdom} for the statement of the relations, and \cite[Lemma 2.17]{Shulman:Thesis} for the proof.

\begin{example}
    Because the Burnside Tambara functor $A_G$ is the unit for the box product every Mackey functor is an $A_G$-module is canonical way.  If $M$ is a Mackey functor, this implies that $M(G/G)$ is a module over the Burnside ring $A(G)$.  The action is uniquely determined by the rule
    \[
        [G/K]\cdot m = \tr^G_K\res^G_K(m)
    \]
    for each $K\leq G$. If $M$ is also a $\ul{\bbZ}$-module then, following \Cref{example: kernel of eta}, multiplication by $[G/K]$ is also multiplication by the integer $|G/K|$.  In particular, every $\ul{\bbZ}$-module $M$ has the relation
    \[
        \tr^G_K\res^G_K(m) = |G/K|\cdot m
    \]
    and hence is additively cohomological.
\end{example}

Combining the last example with \Cref{ex: algebras are modules} indicates the proof of one direction in the following proposition.
\begin{proposition}[{\cite[Proposition 16.3]{TW}}]\label{prop: coh iff Zalg}
    A Tambara functor $T$ is additively cohomological if and only if it admits a map of Tambara functors from $\ulZ$.
\end{proposition}

\begin{remark}\label{rem: tensoring with M is right exact}
    The category of Mackey functors is closed symmetric monoidal, meaning that for all Mackey functors $M$ the functor $(-)\boxtimes M$ has a right adjoint, which gives an internal hom for the category of Mackey functors.  We will not use this construction, but we do note that it implies the functor $(-)\boxtimes M$ is a left adjoint, and is therefore right exact.
\end{remark}

Finally, we will make some use the relative tensor product over a Tambara functor $T$. While we have defined modules as left modules, we can use the fact that $T$ is a commutative $\boxtimes$-monoid to see that every left $T$-module is also a right $T$-module in the usual way.

\begin{definition}
    Let $T$ be a Tambara functor, and let $M$ and $N$ be two $T$-modules.  The relative box product $M\boxtimes_TN$ of $M$ and $N$ over $T$ is defined as the coequalizer of the maps
    \[
    \begin{tikzcd}
            M\boxtimes T\boxtimes N \ar[r,shift left,"\rho\boxtimes N"]\ \ar[r,shift right, "M\boxtimes \lambda"']& M\boxtimes N
    \end{tikzcd}
    \]
    where $\rho$ and $\lambda$ are the right and left acitons of $T$ on $M$ and $N$, respectively.
\end{definition}

\begin{remark}\label{rem: relative tensoring with M is right exact}
    Since the box product is right exact, the relative box product is right exact.
\end{remark}

\section{A reduction}\label{sec: reduction}

In this section we take the first steps toward proving \cref{thm:main}, that every multiplicatively cohomological Tambara functor is also additively cohomological. We explain how proving \cref{thm:main} can be reduced to checking that a single Tambara functor is additively cohomological. This reduction already appears in \cite[Corollary 61]{BQS}, although for the sake of completeness, and establishing notation, we recall the details.  The key idea is that any Tambara functor admits universal quotients which are multiplicatively or additively cohomological.  It suffices to prove that in the case of $A_G$ these two quotients actually agree. We being by recalling the construction of these quotients for $A_G$.

Let $M_G\subset A_G$ denote the smallest Tambara ideal such that $M_G(G/K)$ contains every element of the form $x^{|K/H|}-\nm^K_H\res^K_H(x)$, where $x$ runs over all elements of $A_G(G/K)$.  We write $A_G^{\mathrm{mult}}$ for the quotient Tambara functor $A_G/M_G$. By construction, $A_G^{\mult}$ is multiplicatively cohomological. Let $I_G\subset A_G$ denote the Tambara ideal given by $I_G(G/K) = \ker(\res^K_e)$. By \Cref{example: kernel of eta} we have  $A_G/I_G\cong \ul{\mathbb{Z}}$.

\begin{lemma}
    There is an inclusion of ideals $M_G\subseteq I_G$.
\end{lemma}
\begin{proof}
    Since $\ulZ$ is multiplicatively cohomological, we see that $M_G$ is contained in the kernel of the unit map $A_G\to \ulZ$, which is $I_G$.
\end{proof}

The lemma implies that we have a sequence of maps of Tambara functors
\[
    A_G\xrightarrow{p} A_G^{\mult}\xrightarrow{q} \ulZ,
\]
where each map is surjective.  In the next section, we prove the following theorem.
\begin{theorem}\label{thm:main2}
    The map $q$ is an isomorphism.
\end{theorem}

Assuming this, we can prove \cref{thm:main}.

\begin{proof}[Proof of \cref{thm:main}]
    Let $S$ be any multiplicatively cohomological $G$-Tambara functor.  The kernel of the unique map $A_G\to S$ contains all the generators of $M_G$, and thus factors through $A_G^{\mult}$.  Thus, there exists a map of Tambara functors
    \[
        \ulZ\xrightarrow{q^{-1}} A_G^{\mult}\to S
    \]
    and together wtih \cref{prop: coh iff Zalg} we see that $S$ is additively cohomological.
\end{proof}
\begin{remark}
    It is not true that being multiplicatively cohomological is equivalent to being a module over $A_G^{\mult}$. Indeed, if this were true then \cref{thm:main2} would imply that being additively and multiplicatively cohomological are equivalent.  This is not the case, as we saw in \Cref{example: add not mult}.
\end{remark}

\section{Proof of \cref{thm:main2}}\label{sec: 4}

We set our sights toward proving \cref{thm:main2}. Since the map $q$ is surjective it suffices to prove that $\ker(q)$ is the zero ideal. Towards a contradiction, we suppose there exists finite groups for which the theorem is not true.  Then there must exist a finite group $G$ of minimal order for which the theorem is not true.  Thus, we know that \cref{thm:main2}, and hence also \cref{thm:main}, are true for all proper subgroups of $G$.

\begin{remark}\label{remark: not a p group}
    By \cite[Theorem 59]{BQS} we know that $G$ is a not a nilpotent group.  In particular, we will use later that $G$ is not a $p$-group for any prime $p$.
\end{remark}

The first step in analyzing $\ker(q)$ is to see what it looks like at proper subgroups.

\begin{lemma}\label{lemma: main 2 true on proper subgroups}
    For every proper subgroup $K<G$, we have
    \[
        q_K\colon A_G^{\mult}(G/K)\to\mathbb{Z}
    \]
    is an isomorphism.
\end{lemma}
\begin{proof}
    The quotient $p\colon A_G\to A_G^{\mult}$ restricts to a quotient map
    \[
        A_K\cong R^G_K(A_G)\to R^G_K(A_G^{\mult}).
    \]
    Since the target is multiplicatively cohomological, as this property is conserved by the restriction functor, the target is also additively cohomological, by the minimality of $G$.  It follows that this surjection factors through the $K$-Tambara functor $A^{\mathrm{mult}}_K\cong \ulZ$, and hence there is a surjective ring map 
    \[
        \mathbb{Z}\cong \ul{\bbZ}(K/K)\to R^G_K(A_G^{\mult})(K/K)\cong A_G^{\mult}(G/K).
    \]
    This implies that $A_G^{\mult}(G/K)\cong \mathbb{Z}/(p)$ for $p$ some prime or $0$. Of course, we already know that $q_k\colon A_G^{\mult}(G/K)\to \mathbb{Z}$ is a surjection, which rules out everything except $p=0$. Hence $q_K$ is a surjection $\mathbb{Z}\to \mathbb{Z}$, and is therefore an isomorphism. 
\end{proof}

\begin{corollary}\label{cor: kernel is nilpotent}
    The ideal $\ker(q)$ has
    \[
        \ker(q)(G/K)=0
    \]
    for all proper subgroups $K<G$. Moreover, every element in $\ker(q)(G/G)$ is nilpotent.
\end{corollary}
\begin{proof}
    The first claim is immediate from the preceding lemma.  The claim that every element in $\ker(q)(G/G)$ is nilpotent follows from \cref{lemma: kernel of res is nilpotent}, as every element in $\ker(q)(G/G)$ is in the kernel of every restriction map, since $\ker(q)(G/K)=0$ for all proper subgroups $K$.
\end{proof}

Thus, it remains to show that $\ker(q)(G/G)=0$ as well. Let
\[
    \tau = \sum\limits_{H<G} \mathrm{im}(\tr^G_H\colon A_G^{\mult}(G/H)\to A_G^{\mult}(G/G)),
\]
which is is a ring ideal of $A_G^{\mathrm{mult}}(G/G)$ by Frobenius reciprocity.

\begin{proposition}\label{prop: tau annihilates}
    For any $x\in \ker(q)(G/G)$ and any $y\in \tau$ we have $xy=0$.
\end{proposition}
\begin{proof}
    It suffices to prove that for any $y\in A_G^{\mult}(G/H)$, with $H<G$ a proper subgroup, we have $\tr^G_H(y)\cdot x=0$.  Using Frobenius reciprocity we have
    \[
        \tr^G_H(y)\cdot x = \tr^G_H(y\cdot \res^G_H(x))
    \]
    but $\res^G_H(x)\in \ker(q)(G/H)$ and hence $\res^G_H(x)=0$ by the previous corollary.
\end{proof}

To finish the proof of \cref{thm:main2}, it suffices to prove that $1\in \tau$.  For generic elements $x\in A_G(G/G)$ we will write $\ol{x}$ for the class in $A_G^{\mult}(G/G)$. We will not do this with integers.
\begin{lemma}\label{lemma: product in AGMult}
    Let $K<G$ be a proper subgroup and let $H\leq K$.  In $A_G^{\mult}(G/G)$ we have $\ol{G/H} = |K/H|\cdot \ol{G/K}$.
\end{lemma}
\begin{proof}
    We have 
    \[
        \ol{G/H} = \ol{\tr^G_H(1)} = \ol{\tr^G_K\tr^K_H(1)} = \ol{\tr^G_K(|K/H|)} = |K/H|\cdot \ol{\tr^G_K(1)} = |K/H|\cdot \ol{G/K}
    \]
    where the third equality uses the fact that we know $A_G^{\mult}$ is additively cohomological on all proper subgroups, by \cref{lemma: main 2 true on proper subgroups}.
\end{proof}

\begin{corollary}\label{cor: G/K products}
    Let $K<G$ be a proper subgroup, and let $X$ be any virtual finite $G$-set.  We have $\overline{G/K}\cdot \ol{X} = |X|\cdot \overline{G/K}$.
\end{corollary}
\begin{proof}
    The claim for virtual finite $G$-sets follows immediately once is have been established for actual finite $G$-sets, so assume $X$ is an actual finite $G$-set. Using Mackey's double coset formula we see that $(G/K)\times X$ is a finite $G$-set which can be written as the union of orbits $G/H$, where $H\leq K$.  The claim now follows from counting the number of elements on either side of the equation and applying \cref{lemma: product in AGMult}.
\end{proof}

\begin{proposition}\label{prop: maximal proper torsion}
    For any proper subgroup $K<G$, the element $|G/K|-\ol{G/K}$ is torsion.  More precisely, there is an integer $n$ such that  $|G/K|^{n}\cdot(|G/K|-\overline{G/K})=0$.
\end{proposition}
\begin{proof}
    By \cref{cor: kernel is nilpotent}, we know that $|G/K|-\ol{G/K}$ is nilpotent, so there is some sufficiently large $n$ so that $(|G/K|-\ol{G/K})^{n+1}=0$.  Expanding this out, we have
    \[
        0 =(|G/K|-\ol{G/K})^{n+1} = |G/K|^{n+1}-\overline{G/K}\cdot \ol{X} 
    \]
    where $X$ is some virtual finite $G$-set of size $|G/K|^{n}$. Applying \cref{cor: G/K products}, we obtain
    \[
    0 = |G/K|^{n+1}-\overline{G/K}\cdot \ol{X}  = |G/K|^{n+1}-|G/K|^{n}\cdot\overline{G/K} = |G/K|^{n}\cdot(|G/K|-\overline{G/K})
    \]
    proving the proposition.
\end{proof}
\begin{corollary}\label{cor: the useful one}
    For any proper subgroup $K<G$, there is some integer $N$ such that $|G/K|^N$ is in the image of $\tr^G_K$.
\end{corollary}
\begin{proof}
    Pick an $n$ large enough so that $|G/K|^n\cdot(|G/K|-\ol{G/K})=0$.  Then we  have 
    \[
        \tr^G_K(|G/K|^n) = |G/K|^n\cdot \tr^G_K(1) = |G/K|^n\cdot \ol{G/K} = |G/K|^{n+1}
    \]
    where the last equality uses the assumption on $n$.
\end{proof}

Finally we can give the proof of \cref{thm:main2}, which completes the proof of \cref{thm:main}.
\begin{proof}[Proof of \cref{thm:main2}.]
    By \cref{prop: tau annihilates}, it suffices to prove that $1\in \tau$. Let $p_1,\dots,p_n$ be all the distinct prime numbers which divide $|G|$, and let $P_1,\dots,P_n$ be the corresponding Sylow subgroups.  Note that by \cref{remark: not a p group}, we may assume  $n>1$. By \cref{cor: the useful one}, there are positive integers $N_i$ such that $|G/P_i|^{N_i}\in \tau$ for each $i$.  Since the the gcd of the list $|G/P_1|^{N_1},\dots,|G/P_n|^{N_n}$ is $1$, we must have that $1\in \tau$.
\end{proof}

\section{Jacobson radicals and Nakayama's lemma}\label{sec: nakayama}

In this section we explain how being multiplicatively cohomological is related to the Jacobson radical and Nakayama's lemma.  Recall that the Jacobson radical of a commutative ring $R$, which we denote by $J(R)$, is the intersection of all maximal ideals.  Before turning to Tambara functors, we begin by looking at the related notion of $G$-rings, by which we mean a ring with an action by $G$ through ring automorphisms. If $R$ is a $G$-ring, and $I\leq R$ is an ideal, we say that $I$ is $G$-invariant if for all $x\in I$ and $g\in G$ we have $g\cdot x\in I$. An ideal $\mf{m}\subset R$ is \textit{maximal $G$-invariant}, if there is no proper $G$-invariant ideal $L$ with $I\subsetneq L$.

\begin{proposition}\label{prop: maximal G ideals}
    Let $R$ be a $G$-ring. There is a bijection between maximal $G$-invariant ideals of $R$ and orbits of maximal ideals of $R$.  This bijection is made explicit by sending a maximal ideal $\mf{m}\subset R$ to $\bigcap_{g\in G} g\cdot  \mf{m}$.
\end{proposition}
\begin{proof}
    This is a consequence of \cite[Proposition 6.4]{CMQSV:24}.  In the notation of \textit{loc.\ cit.\ } we have that maximal $G$-invariant ideals are the maximal elements of the poset $\Spec_G(R)$, whereas $G$-orbits of maximal ideals of $R$ are the maximal elements of $\Spec(R)/G$.  Since these two posets are in order preserving bijection, via the map given in the statement, the claim follows.
\end{proof}

\begin{corollary}\label{cor: jac of G ring}
    The intersection of all maximal $G$-invariant ideals in a $G$-ring $R$ is the Jacobson radical of $R$.
\end{corollary}
\begin{proof}
    By \cref{prop: maximal G ideals}, the intersection of the maximal $G$-invariant ideals is the same as the of all maximal ideals of $R$, which is the Jacobson radical of $R$.
\end{proof}

We now turn to studying the Jacobson radical of a Tambara functor $T$. An ideal $\mf{m}\subset T$ is maximal if $\mf{m}\neq T$ and $\mf{m}\subseteq I$ implies $\mf{m}=I$ for any other ideal $I\subsetneq T$. 

\begin{definition}
    The \emph{Jacobson radical} of a Tambara functor $T$, denoted $J(T)$, is the intersection of all maximal Tambara ideals.
\end{definition}

We begin with some basic facts about $J(T)$.

\begin{lemma}\label{lemma: functoriality}
    Let $f\colon T\to S$ be a surjective map of Tambara functors.  Then $J(T)\subseteq f^{-1}(J(S))$.
\end{lemma}
\begin{proof}
    Let $I$ be the kernel of $f$.  Then the poset of ideals of $S$ are in bijection, via $f^{-1}$, with the poset of ideals which contain $I$. Thus we have the following chain of inclusions:
    \[
        J(T) = \bigcap_{\mf{m}\subset T \textrm{ maximal}} \mf{m}\subseteq \bigcap_{I\leq \mf{m}\subset T \textrm{ maximal}} \mf{m} = \bigcap_{\mf{p}\subset S \textrm{ maximal}}f^{-1}(\mf{p}) = f^{-1}(J(S)).
    \]
\end{proof}

\begin{lemma}\label{lemma: maximal ideal looks like this}
    Let $\mf{m}\subset T$ be a maximal ideal.  Then for any $K\leq G$ we have $\mf{m}(G/K) = (\res^K_e)^{-1}(\mf{m}(G/e))$.
\end{lemma}
\begin{proof}
    Let $\mf{p}(G/K) = \res^K_e)^{-1}(\mf{m}(G/e))$. Then $\mf{p}$ is a proper Tambara ideal, and contains $\mf{m}$, thus the two are equal by the maximality of $\mf{m}$.
\end{proof}

\begin{lemma}\label{lemma: bottom of maximal ideal}
    If $\mf{m}\subset T$ is a maximal ideal then $\mf{m}(G/e)\subset T(G/e)$ is a maximal $G$-invariant ideal.
\end{lemma}
\begin{proof}
    Given a proper $G$-invariant ideal $I\subset T(G/e)$ which contains $\mf{m}$, let $\mf{p}(G/K) =(\res^K_1)^{-1}(I)$ for each $H\leq G$. Then $\mf{p}$ is a proper ideal of $T$, and
    \[
        \mf{m}(G/K)=(\res^K_1)^{-1}(\mf{m}(G/e))\subseteq (\res^K_1)^{-1}(I) = \mf{p}(G/K)
    \]
    for all $H$.  Maximality of $\mf{m}$ then implies that $\mf{m}=\mf{p}$, and hence $\mf{m}(G/e) = \mf{p}(G/e) = I$.
\end{proof}

\begin{corollary}\label{cor: bij of max ideals}
    Thee assignment $\mf{m}\mapsto \mf{m}(G/e)$ is a bijection between maximal ideals of $T$ and maximal $G$-invariant ideals of $T(G/e)$.
\end{corollary}
\begin{proof}
    Immediate from the the last two lemmas.
\end{proof}

\begin{corollary}\label{cor: jac is res inverse of jac}
    For any Tambara functor $T$, the Jacobson radical is given by $J(T)(G/K) = (\res^K_e)^{-1}(J(T(G/e)))$.
\end{corollary}
\begin{proof}
    For any $K\leq G$ we have
    \begin{align*}
        J(T)(G/K) & = \bigcap_{\mf{m}} \mf{m}(G/K)  \\ &  = \bigcap_{\mf{m}} (\res^K_e)^{-1}(\mf{m}(G/e)) \\
        & = (\res^K_e)^{-1}\left(\bigcap_{\mf{m}} \mf{m}(G/e) \right) \\ & = (\res^K_e)^{-1}(J(T(G/e))).
    \end{align*}
    where the intersections are over all maximal ideals $\mf{m}\subset T$. The last equality uses \cref{cor: bij of max ideals,cor: jac of G ring}.
\end{proof}

\begin{example}\label{ex: Jac of Burnside}
    When $G= C_p$, for $p$ a prime, the Burnside Tambara functor is presented by the Lewis diagram
\[\begin{tikzcd}[ampersand replacement=\&]
	{\mathbb{Z}[x]/(x^2-px)} \\
	\\
	{\mathbb{Z}}
	\arrow["{\res^{C_p}_e}"{description}, from=1-1, to=3-1]
	\arrow["{\tr^{C_p}_e}", shift left=3, curve={height=-12pt}, from=3-1, to=1-1]
	\arrow["{\nm^{C_p}_e}"', shift right=2, curve={height=12pt}, from=3-1, to=1-1]
\end{tikzcd}\]
with operations defined by
\[
    \res^{C_p}_e(a+bx) = a+pb,\quad \tr^{C_p}_e(k) = kx,\quad \nm^{C_p}_e(k) = k+\frac{k^p-k}{p}x.
\]
Since the Jacobson radical of $\mathbb{Z}$ is zero, we see  that $J(A_{C_p})$ is the ideal given levelwise by $J(A_{C_p})(C_p/C_p) = (x-p)$, and $J(A_{C_p})(C_p/e) = 0$.  
\end{example}

We now turn to the relationship between the ideals $J(T)(G/H)$ and $J(T(G/H))\subset T(G/H)$. Below, we show that the latter is always contained in the former.  First, we need a lemma.

\begin{lemma}
    Let $T$ be a $G$-Tambara functor with injective restrictions.  For all $H\leq K$ the map $\res^K_H\colon T(G/K)\to T(G/H)$ is an integral extension.  That is, every element in $T(G/H)$ satisfies a monic polynomial equation with coefficients in $\mathrm{im}(\res^K_H)$.
\end{lemma}
\begin{proof}
    Since we have chains of ring inclusions
    \[
        T(G/K)\xrightarrow{\res^K_H} T(G/H)\xrightarrow{\res^H_e} T(G/e)
    \]
    it suffices to prove that $\res^H_e\circ \res^K_H = \res^K_e$ is an integral extension for all $K$.  This map factors as 
    \[
        T(G/K)\xrightarrow{\res^K_e} T(G/e)^K\to T(G/e)
    \]
    where the second map is the inclusion of fixed points  (see \cref{remark: Weyl group actions}).  Since the second map is well known to be an integral extension for any finite group, it suffices to prove that the first map is an integral extension.  If $t\in T(G/e)^K$ is any fixed point then the multiplicative double coset formula (\cref{prop: compatibility} (7)) tells us that 
    \[
        \res^K_e\nm^K_e(t) = \prod_{k\in K} k\cdot t = t^{|K|}
    \]
    and thus $t$ satisfies the monic polynomial $x^{|K|}-\res^K_e\nm^K_e(t)$.
\end{proof}

\begin{proposition}\label{prop: jac inclusion}
    Let $T$ be any Tambara functor with Jacobson radical $J(T)$.  We have $J(T(G/H))\subseteq J(T)(G/H)$.
\end{proposition}
\begin{proof}

    For the sake of readability we will, for this proof only, use the notation $T^K:=T(G/K)$, and similarly for ideals. We will also write $J = J(T)$.  With these notations, we must show that $J(T^H)\subseteq J^H$ for all $H\leq G$.

     We will first give a proof in the case where $T$ has injective restrictions.  In this case, the fact that $T^H\to T^e$ is an integral extension implies, using the going up theorem, that  
    \[
        J(T^H)=(\res^{H}_e)^{-1}(J(T^e)) = J^H.
    \]
    where the second equality is \cref{cor: jac is res inverse of jac}.

    For general $T$, let $K\leq T$ be the ideal given by $K^H = \ker(\res^H_e)$. By construction, $T/K$ is a Tambara functor with injective restriction maps, and thus for all $H$ we have 
    \begin{equation}\label{eq 1}
        J(T^H/K^H) = J((T/K)^H) \subseteq J(T/K)^H
    \end{equation}

    Now, let $q\colon T\to T/K$ be the quotient map.  Since $q$ is surjective, the version of \cref{lemma: functoriality} for commutative rings tells us that
    \begin{equation}\label{eq 2}
         J(T^H)\subseteq q^{-1}(J(T^H/K^H)).
    \end{equation}
        Furthermore, we have 
    \[
      q^{-1}(J(T^H/K^H))\subseteq q^{-1}(J(T/K)^H) \subseteq  J(T)^H+K^H\subseteq J(T)^H
    \]
    where the first inclusion is $q^{-1}$ applied to \eqref{eq 1} and the last is \cref{cor: jac is res inverse of jac}. Combining the last display with \eqref{eq 2} gives $J(T^H)\subseteq J(T)^H$ as desired.
\end{proof}

The next corollary shows that what $T$ is multiplicatively cohomological these two ideals agree.  For this, recall that an element $x$ in a commutative ring $R$ is the Jacobson radical $J(R)$ if and only if $1+xy$ is a unit for all $y\in R$.

\begin{corollary}\label{cor: jac of mc}
    Suppose that $T$ is multiplicatively cohomological.  Then $J(T)(G/H)$ is the Jacobson radical of $T(G/H)$.
\end{corollary}
\begin{proof}

    By the previous lemma we have  $J(T(G/H))\subseteq J(T)(G/H)$ and so it suffices to prove the reverse inclusion. Given $x\in J(T)(G/H)$, we have that $\res^H_e(x)$ is in $J(T(G/e))$.  Thus, for any $y\in T(G/H)$ we have $1+\res^H_e(x)\res^H_e(y) = \res^H_e(1+xy)$ is a unit.  But then
    \[
        (1+xy)^{|H|} = \nm^H_e\res^H_e(1+xy)
    \]
    is a unit, because $\nm^H_e$ is multiplicative, so $1+xy$ is a unit for all $y$.  Thus, $x\in J(T(G/H))$.
\end{proof}
\begin{remark}
    We note that \Cref{example: add not mult} gives an example of a Tambara functor $T$ which is not multiplicatively cohomological but for which the conclusion of \cref{cor: jac of mc} holds.  
\end{remark}
\begin{remark}
    The conclusion of \cref{cor: jac of mc} also holds for any Tambara functor with injective restrictions; this follows from the proof of \cref{prop: jac inclusion}.  This is a strictly weaker result, as every Tambara functor with injective restrictions is multiplicatively cohomological.
\end{remark}

We now turn to the proof of Nakayama's lemma. Recall that the statement of Nakayama's lemma is that for any commutative ring $R$, the only $R$-module $M$ with $J(R)M=M$ is $M=0$.  To state Nakayama's lemma for Tambara functors, we need to specify what plays the role of $J(R)M$.  
\begin{definition}
    Let $T$ be a $G$-Tambara functor, $I\subseteq T$ an ideal,  and $M$ a $T$-module.  Then $IM\subseteq M $ is the image of the composite
    \[
        I\boxtimes M\to T\boxtimes M\to M
    \]
    where the first map is induced by the inclusion $I\to T$ and the second is the action map. 
\end{definition}

\begin{example}\label{ex: nak fails in general}
    Let $T = A_{C_p}$ be the Burnside Tambara functor for the group $C_p$, and let $I = J(A_{C_p})$ be the Jacobson radical which is computed in \Cref{ex: Jac of Burnside}.  Let $M$ be the Mackey functor with $M(C_p/C_p) = \mathbb{Z}/(p+1)$ and $M(C_p/e)=0$; the restriction and transfer maps are automatically zero.  The map
    \[
        J(A_{C_p})\boxtimes M\to M
    \]
    is entirely determined by what happens at the $C_p/C_p$ level, where it is the map
    \[
       \mathbb{Z}/(p+1)\cong  \mathbb{Z}\{x-p\} \otimes \mathbb{Z}/(p+1)\to \mathbb{Z}/(p+1)
    \]
    given by multiplication by $x-p$.  If $k\in M(C_p/C_p)=\mathbb{Z}/(p+1)$, we have
    \[
        (x-2)\cdot k = xk-2k = \tr^{C_p}_e\res^{C_p}_e(k)-pk =-pk,
    \]
    where the last step uses that the restriction is zero in $M$.  Since $p$ is invertible mod $p+1$, this map is surjective and we see that $J(A_{C_p})M=M$.
\end{example}

The last example shows that Nakyama's lemma can fail, in general.  The primary obstruction is the Jacobson radical of a Tambara functor is not the levelwise Jacobson radical, hence we cannot rely on Nakayama's lemma for ordinary commutative rings.  However, when $T$ is multiplicatively cohomological we can apply \cref{cor: jac of mc} and prove Nakayama's lemma.

\begin{theorem}[Nakayama's lemma]\label{thm:nakayama}
    Suppose that $T$ is a Tambara functor such that $J(T)(G/H)$ is the Jacobson radical of $T(G/H)$ for all $H$.  If $M$ is any $T$-module such that $J(T)M=M$ then $M=0$.  In particular, this holds when $T$ is multiplciatively cohomological.
\end{theorem}
\begin{proof}
    Toward a contradiction, suppose there is an $M$ for which this is not true.  Let $H\leq G$ be any subgroup which is minimal with respect to $M(G/H)\neq 0$.  Then by \cref{cor: box product whe stuff vanishes}, we have that $(J(T)\boxtimes M)(G/H)$ is some quotient of $J(T)(G/H)\otimes M(G/H)$.   Thus, $(J(T)\boxtimes M)(G/H)$  is a quotient of $J(T(G/H))\otimes M(G/H)$.  Since $J(T)M=M$, the map $J(T(G/H))\otimes M(G/H)\to M(G/H)$ is surjective and  ordinary Nakayama's lemma implies that $M(G/H)=0$, which is a contradiction. The final claim follows from \cref{cor: jac of mc}.
\end{proof}

\begin{remark}
    The statement of \cref{thm:nakayama} might reasonably lead one to use the levelwise Jacobson radical as the definition of $J(T)$, and hence try to avoid the need for further assumptions.  Unfortunately, it seems that such a definition is somewhat unhelpful.  First, it is not at all clear that the levelwise Jacobson radicals form a Tambara ideal, as it is not clear they are closed under any of transfer, norm, or restriction.  Second, in applications one is often interested in the case of local Tambara functors where the Jacobson radical is identified with the unique maximal ideal.  This identification is clear when $J(T)$ is defined as the intersection of maximal ideals, and thus one still wants to identify $J(T)$ with the levelwise Jacobson radical.
\end{remark}

We also give the following standard variation which is often convenient.
\begin{corollary}\label{cor: nakayama  v2}
    Suppose that $T$ is multiplicatively cohomological Tambara functor, $J = J(T)$ is the Jacobson radical, and $M$ is a $T$-module. If $N\subseteq M$ is such that $M = N+JM$ then $N=M$.
\end{corollary}
\begin{proof}
    It suffices, by Nakayama's lemma, to prove that $J\cdot (M/N) = M/N$.  To that end, we apply the snake lemma to the map of exact sequences
\[\begin{tikzcd}[ampersand replacement=\&]
	\& {J\boxtimes N} \& {J\boxtimes M} \& {J\boxtimes (M/N)} \& 0 \\
	0 \& N \& M \& {M/N} \& 0
	\arrow[from=1-2, to=1-3]
	\arrow[from=1-2, to=2-2]
	\arrow[from=1-3, to=1-4]
	\arrow[from=1-3, to=2-3]
	\arrow[from=1-4, to=1-5]
	\arrow[from=1-4, to=2-4]
	\arrow[from=2-1, to=2-2]
	\arrow[from=2-2, to=2-3]
	\arrow[from=2-3, to=2-4]
	\arrow[from=2-4, to=2-5]
\end{tikzcd}\]
to give a short exact sequence of cokernels
\[
    N/JN\to M/JM\to (M/N)/J\cdot(M/N)\to 0.
\]
It suffices, then, to prove that $N/JN\to M/JM$ is surjective.  But this is clear from the fact that $M = N+JM$. Explicitly, for each finite $G$-set $X$ and $x\in M(X)$ there exists a $y\in N(X)$ and a $z\in (JM)(X)$ such that $x = y+z$.  Then 
\[
    x+JM(X) = y+JM(X)
\]
is the image of the class $y+JN(X)$ under the map $N/JN\to M/JM$. Since a map of Tambara modules is surjective if and only it is levelwise surjective, the proof is complete.
\end{proof}

\section{Applications to local Tambara functors}\label{sec: local}

We end the paper with some applications of Nakayam's lemma for multiplicatively cohnmolocial Tambara functors. Recall that a commutative ring is local is it has exactly one maximal ideal.

\begin{definition}
A \emph{local} Tambara functor is a Tambara functor with precisely one maximal ideal.
\end{definition}

When it is convenient, we will write $(T,\mf{m})$ to denote a local Tambara functor $T$ with maximal ideal $\mf{m}$.  When $(T,\mf{m})$ is local, the quotient $T/\mf{m}$ is a Tambara field, meaning the the zero ideal is the only proper ideal.  We will call this the \emph{residue field} of $T$.

\begin{proposition}\label{prop: local characterization}
    A Tambara functor $T$ is local if and only if $T(G/e)$ has exactly one $G$-orbit of maximal ideals.
\end{proposition}
\begin{proof}
    This is immediate from \cref{cor: bij of max ideals,prop: maximal G ideals}.
\end{proof}

Recall that a commutative ring $R$ is \emph{semi-local} if $R/J(R)$ is a semi-simple ring.  That is, $R/J(R)$ is isomorphic to a product of fields. This is equivalent to $R$ having only finitely many maximal ideals. The following generalizes \cite[Proposition 2.4]{Wisdom:classification}.

\begin{corollary}
    If $T$ is a local Tambara functor then $T(G/e)$ is a semi-local ring.
\end{corollary}
\begin{proof}
   By \cref{prop: local characterization}, $T(G/e)$ has at most $|G|$ maximal ideals.
\end{proof}
\begin{remark}
    The converse to this corollary does not hold.  Indeed, $T(G/e)$ will be semi-local whenever it has finitely many orbits of maximal ideals.
\end{remark}

\begin{example}
    We give an example of a local Tambara functor $T$ where $T(G/e)$ is semi-local but not local.  Let $G = C_2/e$, and consider the Tambara functor $T$ indicated in the Lewis diagram
\[\begin{tikzcd}[ampersand replacement=\&]
	{\mathbb{F}_2} \\
	{\mathbb{F}_2\oplus \mathbb{F}_2}.
	\arrow["\Delta"{description}, from=1-1, to=2-1]
	\arrow["{+}", shift left=3, from=2-1, to=1-1]
	\arrow["\times"', shift right=3, from=2-1, to=1-1]
\end{tikzcd}\]
This Tambara functor has exactly one proper ideal, hence is local, but $T(G/e)$ is not a local ring as it has two maximal ideals given by $0\times \bbF_2$ and $\bbF_2\times 0$.  In this case the maximal Tambara ideal is the zero ideal.  For an example where $\mf{m}\neq 0$, one can write down the same Tambara functor with the $2$-adic integers $\bbZ_{2}$ in place of $\bbF_2$ and obtain a local Tambara functor $S$ with residue field $T$.
\end{example}

An important classical result about local rings is that a finitely generated projective module over a local ring is always free.  In the remainder of this section we will prove an analog of this theorem for local Tambara functors which are multiplicatively cohmological.  The reason for the added condition is that we need Nakayama's lemma, and for this we need to apply \cref{thm:nakayama}.

Before giving the proof, we must first clarify what we mean by a free module.  For commutative rings, free modules are classified by their dimension, which is a positive integer.  For Tambara functors, the notion of freeness we use is indexed on finite $G$-sets.

In the remainder of this section, we will need to use the fact that Tambara and Mackey functors can be extended to take all finite $G$-sets as input.  This is done by requiring that a Tambara functor $T$ come equipped with sufficiently natural isomorphisms
\[
    T(X\amalg Y)\cong T(X)\times T(Y),
\]
where $X$ and $Y$ are any two finite $G$-sets. The key point is that every finite $G$-set is the union of its orbits, and each orbit is isomorphic to one of the $G/H$; thus $T(X)$ is determined, up to isomorphism, by the rings $T(G/H)$. We refer the reader to \cite{tambara:1993} for the rigorous construction of Tambara functors indexed on finite $G$-sets.  

For a Tambara functor $T$, let $\Mod_{T}$ denote the category of left $R$-modules.  Base change along the unit map $A_G\to T$ produces an adjunction
\[\begin{tikzcd}[ampersand replacement=\&]
	{\Mack^G} \&\& {\Mod_T.}
	\arrow["{T\boxtimes(-)}", shift left=2, from=1-1, to=1-3]
	\arrow["U", shift left=2, from=1-3, to=1-1]
\end{tikzcd}\]

Now, for each finite $G$-set $X$ there is a Mackey functor $A_X$ which represents evaluation at the $G$-set $X$.  That is, there is a natural isomorphism of abelian groups $\Mack^G(A_X,M)\cong M(X)$ for each finite $G$-set $X$.

\begin{definition}\label{def: free modules}
    A $T$-module is \emph{finite free} if it is isomorphic to one of the form $T_X:=T\boxtimes A_X$ where $X$ is a finite $G$-set.
\end{definition}
\begin{remark}
    One can extend the definition to non-finite $G$-sets as follows.  Given any $G$-set $Y$, write $Y$ as the union of an increasing chain $X_0\subset X_1\subset\dots$ where each $X_i$ is finite.  Then define $A_Y = \mathrm{colim}_{i}\ A_{X_i}$.  One checks that this colimit does not depend on the choice of $X_i$'s.  In this paper we will only be concerned with finite free modules.
\end{remark}
The name is justified by the following proposition.
\begin{proposition}\label{prop: representability}
    For any $T$-module $M$ there is a natural bijection of abelian groups
    \[
        \Mod_T(T_X,M)\cong M(X)
    \]
    for any finite $G$-set $X$.
\end{proposition}
\begin{proof}
    We have
    \[
    \Mod_T(T_X,M)\cong \Mack^G(A_X,U(M))\cong U(M)(X)=M(X)
    \]
    where the first isomorphism is the free-forgetful adjunction between $\Mack^G$ and $\Mod_T$.
\end{proof}

We say that a $T$-module $M$ is \emph{finitely generated} if it admits a surjection from $T_X$ for some finite $G$-set $X$.
\begin{lemma}
    A $T$-module is finitely generated and projetive if and only if it is a retract of $T_X$ for some finite $G$-set $X$.
\end{lemma}
\begin{proof}
    Representability implies that each $T_X$, and hence all retracts of a $T_X$, are projective.  Conversely, if $P$ is a finitely generated projective module then by definition it admits a surjection $T_X\to P$ since it is finitely generated and this surjection splits because $P$ is projective.
\end{proof}

\begin{lemma}
    Let $I$ be any ideal of a Tambara functor, and let $M$ be a $T$-module.  For any finite $G$-set $X$ there is an isomorphism of abelian groups
    \[
        ((T/I)\boxtimes_{T} M)(X)\cong M(X)/(IM)(X).
    \]
\end{lemma}
\begin{proof}
    This follows by tensoring the exact sequence
    \[
        I\to T\to T/I\to 0
    \]
    with $M$ and using the fact that $(-)\boxtimes_T M$ is right exact, by \cref{rem: relative tensoring with M is right exact}.
\end{proof}

\begin{theorem}
    Let $(T,\mf{m})$ be a local, multiplicatively cohomological Tambara functor. If $P$ is a finitely generated projective $T$-modules such that $P/\mf{m}P$ is a free $T/\mf{m}$-module then $P$ is a free $T$-module.
\end{theorem}
\begin{proof}
    We mimic the ordinary proof for local rings, avoiding the use of elements where possible. We write $F = T/\mf{m}$ for the residue field.  By assumption, $F\boxtimes_{T} P\cong P/\mf{m}P$ is free, hence there is a finite $G$-set $X$ and an isomorphism $h\colon F_X = F\boxtimes A_X\cong P/\mf{m}P$.  Such a map corresponds to an element $x\in (F\boxtimes_T P)(X)\cong (P/\mf{m}P)(X)$.  Let $\overline{x}\in P(X)$ be a lift along the canonical quotient $P(X)\to (P/\mf{m}P)(X)$.  Using \cref{prop: representability} again, this element $\overline{x}$ corresponds to a map $\overline{h}\colon T_X\to P$, and this map satisfies the property that the composite
    \[
        F_X\cong F\boxtimes_T T_X\xrightarrow{F\boxtimes_T \overline{h}} F\boxtimes_T P = P/\mf{m}P
    \]
    is equivalent to $h$.  In particular, $F\boxtimes \overline{h}$ is an isomorphism. Equivalently, for any finite $G$-set $Y$ we have
    \[
        P(Y)\cong \mathrm{im}(\overline{h}\colon T_X(Y)\to P(Y)) + \mf{m}P(Y)
    \]
    which implies, by \cref{cor: nakayama  v2}, that $\overline{h}$ is surjective. It follows that $T_X\cong P\oplus Q$, where $Q$ is some other projective $T$-module and the projection onto $P$ is $\overline{h}$. We will show that $Q=0$, which completes the proof.
    
    Since functors preserve retract diagrams, we see that there is an exact sequence
    \[
        0\to F\boxtimes_T Q\to F\boxtimes_T T_X\xrightarrow{F\boxtimes_T \overline{h}} F\boxtimes_T P\to 0
    \]
    but we know that $F\boxtimes_T \overline{h}$ is an isomorphism and hence $F\boxtimes_T Q=0$.  If one repeats the argument  with $P$ replaced by $Q$, we see there is a surjection from $0= T_{\emptyset} \to Q$, which implies that $Q=0$.
\end{proof}

Of course, the utility of this result relies heavily on the existence of residue fields with the property that every projective module is free.  Luckily, there are some known results here.

\begin{theorem}[{\cite[Theorem E]{ChanWisdom}}]\label{thm: Chan Wisdom}
    Let $F$ be a $C_{p^n}$-Tambara functor such that:
    \begin{itemize}
        \item for all $H\leq G$, $F(G/H)$ is a field,
        \item the restriction maps $\res^K_H\colon F(G/K)\to F(G/H)$ are finite extensions.
    \end{itemize}
    Then a finitely generated $F$-module is projective if and only if it is free.
\end{theorem}

There are many examples.  For instance, the fixed point Tambara functor of any $C_{p^n}$-action on a field $L$.  In particular, if $L$ is any field, with trivial $C_{p^n}$-action, then the constant $C_{p^n}$-Tambara functor at $L$ has this property.

\begin{corollary}
    Let $T$ be any local Tambara functor with residue field satisfying the conditions of \cref{thm: Chan Wisdom}.  Then a finitely generated $T$-module is projective if and only it is free.
\end{corollary}

We give the example which are likely of the most interest, at least in equivariant stable homotopy theory. For a prime $p$, let $\ul{\bbZ}_p$ denote the constant $C_{p^n}$-Tambara functor at the $p$-adic integers.

\begin{theorem}
    A finitely generated $\ul{\bbZ}_p$-module is projective if and only if it is free.
\end{theorem}
\begin{proof}
    Since every fixed point Tambara functor is multiplicatively cohomological, it suffices to prove that $\ul{\bbZ}_p$ is local with residue field satisfying the conditions of \cref{thm: Chan Wisdom}.  The first claim follows from \cref{cor: bij of max ideals}.  For the second claim, we claim that the maximal ideal $\mf{m}\subset T$ is given by $\mf{m}(G/H) = p\bbZ_p\subset \bbZ_p$ for each $H$.  This implies that the residue field is constant at $\bbF_p$, which satisfies the conditions of \cref{thm: Chan Wisdom}.  The claim is clear from \cref{lemma: maximal ideal looks like this}, the fact that each restriction map is an isomorphism, and the fact that $\mf{m}(G/e) = p\bbZ_p$ is the unique maximal ideal of $\bbZ_p$.
\end{proof}
\begin{remark}
    Of course, the same proof applies to the constant $C_{p^n}$-Tambara functor at any local ring.
\end{remark}

\printbibliography

\end{document}